\documentclass[a4paper]{amsart}

\usepackage{amscd,amsthm,amsfonts,amssymb,amsmath}
\usepackage{mathrsfs}
\usepackage[colorlinks=true]{hyperref}
\usepackage{fullpage}
\usepackage[all]{xy}
\usepackage{verbatim}

    \newcommand{\BA}{{\mathbb {A}}} \newcommand{\BB}{{\mathbb {B}}}
    \newcommand{\BC}{{\mathbb {C}}} 
    \newcommand{\BE}{{\mathbb {E}}} \newcommand{\BF}{{\mathbb {F}}}

     \newcommand{\BL}{{\mathbb {L}}}
     
     \newcommand{\BP}{{\mathbb {P}}}
    \newcommand{\BQ}{{\mathbb {Q}}}

     \newcommand{\BZ}{{\mathbb {Z}}}

     \newcommand{\CH}{{\mathcal {H}}}

    \newcommand{\CO}{{\mathcal {O}}}

    \newcommand{\ab}{{\mathrm{ab}}}

    \newcommand{\Aut}{{\mathrm{Aut}}}

    \newcommand{\Char}{{\mathrm{Char}}}

    \newcommand{\End}{{\mathrm{End}}}

    \newcommand{\Gal}{{\mathrm{Gal}}} \newcommand{\GL}{{\mathrm{GL}}}

     \newcommand{\Pic}{\mathrm{Pic}}

    \renewcommand{\mod}{\ \mathrm{mod}\ }

    \newcommand{\Sel}{{\mathrm{Sel}}}

    \newcommand{\tor}{{\mathrm{tor}}}

    \newcommand{\Vol}{{\mathrm{Vol}}}

     \newcommand{\M}{\mathrm{M}}
        \newcommand{\Tr}{\mathrm{Tr}}

    \newcommand{\lrb}[1]{\left(#1\right)} \newcommand{\BetaI}[1]{{\left\{#1 \right\} }}

\newcommand{\matrixx}[4]{\begin{pmatrix}
#1 & #2 \\ #3 & #4
\end{pmatrix} }        

    \font\cyr=wncyr10

    \newcommand{\Sha}{\hbox{\cyr X}}
    \newcommand{\wh}{\widehat}

    \newcommand{\ov}{\overline}

    \newcommand{\lra}{\longrightarrow}
    \newcommand{\ra}{\rightarrow}

                            \newcommand{\hilbert}[4]{\left(\frac{#1,#2}{#3;#4}\right)}

    \theoremstyle{plain}
    \newtheorem{thm}{Theorem}[section] \newtheorem{coro}[thm]{Corollary}
    \newtheorem{lem}[thm]{Lemma}  \newtheorem{prop}[thm]{Proposition}

\theoremstyle{remark} \newtheorem{remark}{Remark}[section]
\theoremstyle{remark} 
\theoremstyle{remark} 

    \numberwithin{equation}{section}

\begin{document}
\title{Cube sums of form $3p$ and $3p^2$ II}
\author{Jie Shu and Hongbo Yin}
\begin{abstract}
Let  $p\equiv 2,5\mod 9$ be a prime. We prove that both $3p$ and $3p^2$ are cube sums. We also establish some explicit Gross-Zagier formulae and investigate the 3 part full BSD conjecture of the related elliptic curves. 
\end{abstract}
\address{School of Mathematical Sciences, Tongji University, Shanghai 200092}
\email{shujie@tongji.edu.cn}
\address{School of Mathematics, Shandong University, Jinan 250100,
P.R.China}

\email{yhb2004@mail.sdu.edu.cn}
\thanks{Jie Shu is supported by NSFC-11701092; Hongbo Yin is supported by NSFC-11701548 and Young Scholar Program of Shandong University.}

\maketitle


\section{Introduction}
We call a nonzero rational number a cube sum if it is of the form $a^3+b^3$ with $a,b\in \BQ^\times$. For the history and background about this Diophiantin problem please refer to \cite{DV1}\cite{DV17}\cite{HSY}\cite{SSY}.  Up to now, only four family numbers with many prime factors are proved to be cube sums \cite{Satge}\cite{Coward} and the Sylvester conjecture concerning primes $p$ only has partial result \cite{DV17}.
In this paper, we mainly prove the following theorem completing our partial result in \cite{SSY} using a different construction.
\begin{thm}\label{main}
Let  $p\equiv 2,5\mod 9$ be a prime. Then both $3p$ and $3p^2$ are cube sums.
\end{thm}

Since $2$ is a cube sum, from now on, we may assume $p\equiv 2,5\mod 9$ is an odd prime number. Let $E_n$ be the elliptic curve given by $x^3+y^3=nz^3$. It has the Weierstrass equation $y^2=x^3-432n^2$. If $n>2$ is not a cube, then $E_n(\BQ)_{\tor}=0$ and $n$ is a cube sum if and only $E_n(\BQ)$ has rank at least one. Following \cite{Satge87}, we use the Heegner points twisted from a fixed elliptic curve to prove the above theorem. 


In second part of this paper, we establish some explicit Gross-Zagier formulae (Theorem \ref{thm:GZ}) and use them to investigate the 3-part full BSD conjecture for $E_{3p}$ and $E_{3p^2}$. 
More explicitly,
let $\Sha(E_n)$, $E_n(\BQ)_\tor$, $\Omega_n$, $R(E_n)$ and $c_\ell(E_n)$ denote the Shafarevich-Tate group, the torsion subgroup,  the minimal real period, the regulator and the Tamagawa number of $E_n$ over $\BQ$ respectively. 
Then the full BSD conjecture predicts that if $L(s,E)$ is of order $r$ at $s=1$, then
\begin{equation*}
 |\Sha(E_{n})|=\frac{L^{(r)}(1,E_{n})}{\Omega_{n}\cdot R(E_n)}\cdot \frac{ |{E_{n}}(\BQ)_\tor|^2}{\prod_\ell c_\ell(E_{n})}.
\end{equation*}
Let $P$ (resp. $Q$) be a generator of the free part of $E_{3p^2}(\BQ)$(resp. $E_{3p}(\BQ)$). We prove that
\begin{thm}\label{main2}
Let $p\equiv 2 \mod 9$ be a rational prime number. Then 
\begin{equation}\label{bsd1}
|\Sha(E_p)|\cdot|\Sha(E_{3p^2})|=\frac{L(1,E_p)}{\Omega_p\cdot \widehat{h}_\BQ(P)}\cdot\frac{L'(1,E_{3p^2})}{\Omega_{3p^2}}\cdot \frac{ |{E_p}(\BQ)_\tor|^2}{\prod_\ell c_\ell(E_p)}\cdot \frac{ |{E_{3p^2}}(\BQ)_\tor|^2}{\prod_\ell c_\ell(E_{3p^2})},
\end{equation}
up to a power of $2p$.

Let $p\equiv 5 \mod 9$ be a rational prime number. Then 
\begin{equation}\label{bsd2}
|\Sha(E_{p^2})|\cdot|\Sha(E_{3p})|=\frac{L(1,E_{p^2})}{\Omega_{p^2}\cdot \widehat{h}_\BQ(Q)}\cdot\frac{L'(1,E_{3p})}{\Omega_{3p}}\cdot \frac{ |{E_{p^2}}(\BQ)_\tor|^2}{\prod_\ell c_\ell(E_{p^2})}\cdot \frac{ |{E_{3p}}(\BQ)_\tor|^2}{\prod_\ell c_\ell(E_{3p})},
\end{equation}
up to a power of $2p$.
\end{thm}

Note that by the work of Perrin-Riou \cite{PR1987}, Kobayashi \cite{Koba2013}, the $\ell$ part full BSD conjecture of $E_{3p}$ and $E_{3p^2}$ is known for $\ell\nmid 6p$. 
But the prime $3$ is very special in the Iwasawa theory for the elliptic curve family $E_D: y^2=x^3+D$ whose CM field $K=\BQ(\sqrt{-3})$ has $6$ roots of unity and $2$ is special for all elliptic curves.  In particular, there is no any general results about the $2$ and $3$ part full BSD conjecture of $E_D$.

\section{Modular Actions on Heegner Points}
\subsection{Modular curves and modular actions}
We will use the notations as in \cite[Section 2]{HSY} for the related modular curves. 
Recall $X_0(3^5)$ is the
classical modular curve over $\BQ$ of level $\Gamma_0(3^5)$.
Define $N$ to be the normalizer of $\Gamma_0(3^5)$ in $\GL_2^+(\BQ)$.
Then the linear fractional transformation action of $N$ on $X_0(3^5)$ induces an isomorphism
\[N/\BQ^\times\Gamma_0(3^5)\simeq \Aut_{\ov{\BQ}}(X_0(3^5)).\]
The quotient group $N/\BQ^\times\Gamma_0(3^5)\simeq S_3\rtimes \BZ/3\BZ$, where $S_3$ denotes the symmetric group with $3$ letters which is generated by the Atkin-Lehner operator $W=\begin{pmatrix}0&1\\-3^5&0\end{pmatrix}$ and the matrix $A=\begin{pmatrix}28&1/3\\3^4&1\end{pmatrix}$, and the subgroup $\BZ/3\BZ$ is generated by the matrix $B=\begin{pmatrix}1&0\\3^4&1\end{pmatrix}$.

Put
\[U=\langle U_0(3^5),W,A\rangle\subset \GL_2(\BA_f),\]
and $$\Gamma =\GL_2(\BQ)^+\cap U=\langle \Gamma_0(3^5),W,A\rangle,$$ and let $X_\Gamma$ be the
modular curve over $\BQ$ of level $\Gamma$ whose underlying Riemann surface is
\[X_\Gamma(\BC)=\GL_2(\BQ)^+\backslash\left (\CH\sqcup\BP^1(\BQ)\right )\times \GL_2(\BA_f)/U.\]
So the curve is the quotient of $X_0(3^5)$ by the actions of $W$ and $A$.
Then $X_\Gamma$ is a smooth projective curve over $\BQ$ of genus $1$, and the infinity cusp $[\infty]$ is rational over $\BQ$. We identify $X_\Gamma$ with an elliptic curve over $\BQ$ with $[\infty]$ as its zero element \cite[Proposition 2.1]{HSY}. Let $N_\Gamma$ be the normalizer of $\Gamma$ in $\GL_2(\BQ)^+$. Then we have a natural embedding
\[\Phi: N_\Gamma/ \BQ^\times\Gamma\hookrightarrow \Aut_{\ov{\BQ}}(X_\Gamma)\simeq \CO_K^\times\ltimes X_\Gamma(\ov{\BQ}),\]
where $\CO_K^\times $ embeds into $\Aut_{\ov{\BQ}}(X_\Gamma)$ by complex multiplications and $X_\Gamma(\ov{\BQ})$ embeds into $\Aut_{\ov{\BQ}}(X_\Gamma)$ by translations.
The matrices
\[B=\begin{pmatrix}1&0\\3^4&1\end{pmatrix},\quad C=\begin{pmatrix}1&1/9\\-3^3&-2\end{pmatrix}\]
lie in $N_\Gamma$, and hence induce automorphisms of $X_\Gamma$.

The elliptic curves $E_n$ are all endowed with complex multiplication by $K$ and we fix the complex multiplication $[\cdot]:\CO_K\simeq \End_K(E_n)$ by $[-\omega](x,y)=(\omega x,-y)$. We will always take the simple Weierstrass equation $y^2=x^3-2^4\cdot 3$ for the elliptic curve $E_9$.  We quote \cite[Proposition 2.1]{HSY} as follows.
\begin{prop}\label{modular-curve}
The elliptic curve $(X_\Gamma,[\infty])$ is isomorphic to $E_9$ over $\BQ$.
Moreover,  for any point $P\in X_\Gamma$, we have
\[\Phi(B)(P)=[\omega^2]P,\quad \Phi(C)(P)=[\omega^2]P+(0,4\sqrt{-3}).\]
 In particular, the automorphisms $\Phi(B)$ and $\Phi(C)$ are defined over $K$.

\end{prop}
Note that there exists a unique isomorphism $X_\Gamma\ra E_9$ over $\BQ$ such that the cusp $[1/9]$ has coordinates $(0,4\sqrt{-3})$. We use this isomorphism to identify $X_\Gamma$ with $E_9$.

Let $V\subset U_0(3^5)$  be the subgroup consisting of matrices
$\left (\begin{array}{cc}a&b\\c&d
\end{array} \right )$ with $a\equiv d\mod 3$, and put $U_0=\langle V,W,A \rangle$.  Let $X_\Gamma^0$ be the modular curve over $\BQ$ whose underlying Riemann
surface is
\[X_\Gamma^0(\BC)=\GL_2(\BQ)^+\backslash\left (\CH\bigsqcup\BP^1(\BQ)\right )\times \GL_2(\BA_f)/U_0.\]
The  modular curve $X_\Gamma^0$ is isomorphic to $X_\Gamma\times_\BQ K$ as a curve over $\BQ$.  Usually, we denote by $[z,g]_{U_0}$  the point  on  $X_\Gamma^0$ which is represented by the pair $(z,g)$ where $z\in \CH$ and $g\in \GL_2(\BA_f)$. Let $N_{\GL_2(\BA_f)}(U_0)$ be the normalizer of $U_0$ in
$\GL_2(\BA_f)$. Then there is a natural homomorphism
$$ N_{\GL_2(\BA_f)}(U_0)/U_0\longrightarrow \Aut_\BQ(X_\Gamma^0)$$ induced by right
translation on $X_\Gamma^0$: for $P=[z,g]_{U_0}\in X_\Gamma^0$ and $x\in N_{\GL_2(\BA_f)}(U_0)$
\[P\mapsto P^x=[z,gx]_{U_0}.\]

\subsection{Modular actions on Heegner points}
Let $p\equiv 2,5\mod 9$ be an odd prime number. Denote 
$$\tau_i=M_i\omega\in \CH,$$
where 
\[M_1=\matrixx{p/9}{0}{2}{1}, \quad M_2=\matrixx{p/9}{0}{5}{1}, \quad M_3=\matrixx{p/9}{0}{2}{4}, \quad \omega=\frac{-1+\sqrt{-3}}{2}.\]
For $i=1,2,3$, let $\rho_i:K\hookrightarrow \M_2(\BQ)$ be the normalised embedding (see \cite{CST17,HSY}) with fixed point $\tau_i\in \CH$.
Then $\rho_i$ are explicitly given by
\[\rho_1(\omega)=M_1\matrixx{-1}{-1}{1}{0}M_1^{-1}=
\begin{pmatrix}1&-p/9\\27/p&-2\end{pmatrix}.\]
\[\rho_2(\omega)=M_2\matrixx{-1}{-1}{1}{0}M_2^{-1}=
\begin{pmatrix}4&-p/9\\187/p&-5\end{pmatrix}.\]
\[\rho_3(\omega)=M_3\matrixx{-1}{-1}{1}{0}M_3^{-1}=
\begin{pmatrix}-1/2&-p/36\\27/p&-1/2\end{pmatrix},\]
Let $R_0(3^5)$ bet the standard Eichler order of discriminant $3^5$ in $\M_2(\BQ)$.  Then $\rho_1(K)\cap R_0(3^5)=\CO_{9p}$, $\rho_2(K)\cap R_0(3^5)=\CO_{9p}$ and $\rho_3(K)\cap R_0(3^5)=\CO_{36p}$. So $\tau_1$, $\tau_2$  is defined over $H_{9p}$ and $\tau_3$ is defined over $H_{36p}$ by the complex multiplication theory. Here $H_m$ is the ring class field of $K$ with conductor $m$. 
\begin{remark}\label{remark1}
In order to prove Theorem \ref{main}, we just need $\tau_1$ and $\tau_2$. But we do not how to get rational points over $\BQ$ from $\tau_1$ and $\tau_2$ since we do not know how the complex conjugation acts on them. In order to prove Theorem \ref{main2}, we need the help of $\tau_3$ which shares the same Galois action with $\tau_2$ but will give us real point directly. This is also the reason why only prove half cases in Theorem \ref{main2}.
\end{remark}

Let $\CO_{K,3}$ be the completion of $\CO_K$ at the unique place above $3$.
Let $\CO_{K,3}$ be the completion of $\CO_K$ at the unique place above $3$. We have
\[ \CO_{K,3}^\times/\BZ_3^\times(1+9\CO_{K,3})=\langle \omega_3\rangle\times\langle1+3\omega_3\rangle\cong \BZ/3\BZ\times\BZ/3\BZ,\]
where $\omega_3$ is the image of $\omega$ into $\CO_{K,3}^\times$. Under both the embeddings $\rho_1$ and $\rho_2$,  it is straightforward to verify that $\omega_3$ and $1+3\omega_3$ normalize $U_0$, and therefore they induce automorphisms of $X^0_\Gamma$. \begin{thm}\label{modular-action}
 Let  $P $ be an arbitrary point on $X_\Gamma^0$.
\begin{itemize}
 \item[1.] 
 Under the embedding $\rho_1$, we have
\[P^{1+3\omega_3}=[\omega]P,\textrm{ and } P^{\omega_3}=[\omega]P+(0,-4\sqrt{-3}).\]

\item[2.] Under the embedding $\rho_2$ and $\rho_3$, we have
\[P^{1+3\omega_3}=[\omega]P,\textrm{ and } P^{\omega_3}=[\omega^2]P+(0,-4\sqrt{-3}).\]
\end{itemize}
\end{thm}
\begin{proof}
We give the proof of the first assertion in details and the case under the embedding $\rho_2$ is similar. Now suppose $K$ is embedded in $\M_2(\BQ)$ under $\rho_1$. Since $\omega_3$ and $1+3\omega_3$ have determinants $\equiv 1\mod 3$, as elements in $\Aut_\BQ(X_\Gamma^0)$, they lie in the subgroup $\Aut_K(X_\Gamma)$. See \cite[page 6]{HSY} for the structure of the automorphism groups.
Suppose $P=[z,1]$, $z\in \CH$, be a point on $X_\Gamma^0$. We have
\[A^2B^2(1+3\omega_3)=\left(\begin{pmatrix}783/p+9508&-2377p/3-145/3\\ 2268/p+27540&-2295p-140\end{pmatrix}_3, A^2B^2\right)\in V,\]
where the subscript $3$ denotes the $3$-adic component of the adelic matrices.
Then by Proposition \ref{modular-curve},
\[P^{1+3\omega_3}=\Phi(B^2)(P)=[\omega]P.\]
Similarly, if $p\equiv 2\mod 9$, then $AC^2\omega_3\in V,$
and hence
\[P^{\omega_3}=\Phi(C^2)(P)=[\omega]P+(0,-4\sqrt{-3}).\]
If $p\equiv 5\mod 9$, then $A^2C^2\omega_3\in V,$
and hence
\[P^{\omega_3}=\Phi(C^2)(P)=[\omega]P+(0,-4\sqrt{-3}).\]

For the case under embbedding $\rho_2$ and $\rho_3$, it is straight to verify that $A^2B^2(1+3\omega_3)\in V$ for any odd prime $p\equiv 2, 5\mod 9$, and $AB^2C^2\omega_3V,$ when $p\equiv 2\mod 9$, and $A^2B^2C^2\omega_3\in V$ when $p\equiv 5\mod 9$. Then the second assertion follows from Proposition \ref{modular-curve}.

\end{proof}

\subsection{Galois actions on Heegner points}
Fix the Artin reciprocity law $\sigma: \wh{K}^\times\ra \Gal(K^\ab/K)$ by sending local uniformizers to Frobenius automorphisms. Denote by $\sigma_t$ the image of $t\in\wh{K}^\times$.   Let $P_i=[\tau_i,1]_{U_0}$ be the CM points on $X_\Gamma^0$ for  $i=1,2,3$. In the following, when we consider the CM point $P_i$, we assume $K$ is embedded in $\M_2(\BQ)$ under $\rho_{i}$.
\begin{thm}\label{thm:Galois}
For $i=1,2$, the point $P_i\in X_\Gamma^0(H_{9p})$ satisfies
\[P_i^{\sigma_{1+3\omega_3}}=[\omega]P_i,\textrm{ and } P_i^{\sigma_{\omega_3}}=[\omega^i]P_i+(0,-4\sqrt{-3}).\]
Similarly, $P_3\in X_\Gamma^0(H_{36p})$ satiesfies
\[P_3^{\sigma_{1+3\omega_3}}=[\omega]P_3,\textrm{ and } P_3^{\sigma_{\omega_3}}=[\omega^2]P_3+(0,-4\sqrt{-3}).\] 
\end{thm}
\begin{proof}
By Shimura's reciprocity law \cite[Theorems 6.31 and  6.38]{Shimurabook}, we have
\[P_i^{\sigma_t}=P_i^t=[\tau_i, t],\quad t\in \wh{K}^\times.\]
Since $\wh{K}^\times\cap U_0=\wh{\CO_{9p}}^\times,$ by class field theory, we see $P_i$ is defined over the ring class field $H_{9p}$, and the Galois actions of $\sigma_{\omega_3}$ and $\sigma_{1+3\omega_3}$ are clear from Theorem \ref{modular-action}. The proof for $P_3$ is similar.
\end{proof}
\begin{remark}\label{remark1}
Since $\tau_3=p\omega/9(2\omega+4)=p\sqrt{-3}/54$, $e^{2\pi i\tau_2/n}$ is real. So $P_3$ is in fact a real point on $E_9$.
\end{remark}

\section{Nontriviality of Heegner points}
The elliptic curve $E_3$ has Weierstrass equation $y^2=x^3-2^43^5$. Consider the isomorphism
\[\phi:E_9\lra E_3,\quad (x,y)\mapsto (9x/\sqrt[3]{9}, 9y).\]
We have the following commutative diagram:
$$\xymatrix{E_9(H_{9p})^{\sigma_{1+3\omega_3}=\omega^2}\ar[d]^\phi\ar[rr]^{\Tr_{H_{9p}/L_{(3,p)}}}&&E_9(L_{(3,p)})^{\sigma_{1+3\omega_3}=\omega^2}\ar[d]^{\phi}\\
            E_3(H_{3p})\ar[rr]^{\Tr_{H_{3p}/L_{(p)}}}&&E_3(L_{(p)})}$$
where the field extension diagram is as follows ($H_m$ is the ring class field of $K$ with conductor $m$):
\[\xymatrix@R=11pt{&H_{9p}=H_{3p}(\sqrt[3]{3})\ar@{-}[dl]^{3}\ar@{-}[dr]\ar@{-}[dd]&\\
            H_{3p}\ar@{-}[dd]_{(p+1)/3}&&H_9\ar@{=}[dd]\\
            &L_{(3,p)}=K(\sqrt[3]{3},\sqrt[3]{p})\ar@{-}[dr]^{3}\ar@{-}[d]^{3}\ar@{-}[dl]_{3}&\\
            L_{(p)}=K(\sqrt[3]{p})\ar@{-}[dr]^{3}&L_{(3p)}=K(\sqrt[3]{3p})\ar@{-}[d]^{3}&L_{(3)}=K(\sqrt[3]{3})\ar@{-}[dl]_{3}\\
            &K\ar@{-}[d]^{2}&\\
            &\BQ.&\\
            }\]
The following proposition on related field extensions is partially quoted from \cite[Proposition 2.6]{SSY}.            
\begin{prop}\label{LCF}
Let $p\equiv 2,5\mod 9$ be odd primes.
\begin{itemize}
\item[1.] The field $H_{9p}=H_{3p}(\sqrt[3]{3})$ with Galois group $\Gal(H_{9p}/H_{3p})\simeq \langle \sigma_{1+3\omega_3}\rangle^{\BZ/3\BZ}$, and
\[\left(\sqrt[3]{3}\right)^{\sigma_{1+3\omega_3}-1}=\omega^2.\]
\item[2.] We have $\left(\sqrt[3]{3}\right)^{\sigma_{\omega_3}-1}=1$ and
\[\left(\sqrt[3]{p}\right)^{\sigma_{\omega_3}-1}=\left\{\begin{aligned} {\omega},&\quad p\equiv 2\mod 9; \\{\omega^2},&\quad p\equiv 5\mod 9.  \end{aligned}\right.\]
\item[3.]  $[H_{36p}:H_{9p}]=6$ and $H_{36p}=H_{9p}(\sqrt{-1},\sqrt[3]{2})$ with Galois group 
$$\Gal(H_{36p}/H_{9p})\simeq \langle \sigma_{1+2\omega_2}\rangle^{\BZ/2\BZ}\times\langle \sigma_{\omega_2}\rangle^{\BZ/3\BZ}.$$ 
\end{itemize}
\end{prop}
\begin{proof}
$(1)$ and $(2)$ are contained in \cite[Proposition 2.6]{SSY}. We just need to prove $(3)$, but the argument is similar to the proof of $(1)$. Note that $\CO_m=\BZ+m\CO_K$,
\[\Gal(H_{36p}/H_{9p})\simeq K^\times \wh{\CO_{9p}}^\times/K^\times\wh{\CO_{36p}}^\times\simeq \wh{\CO_{9p}}^\times/(\wh{\CO_{9p}}^\times\cap K^\times\wh{\CO_{36p}}^\times)\simeq \CO_{K,2}^\times/ \BZ_2^\times(1+4\CO_{K,2})\]
is of order $6$ and generated by $\omega_2$ and $1+2\omega_2$. The ideal $\sqrt{-3}\CO_K=(1+2\omega)$ and let $v$ be the place corresponding to the prime ideal $(1+3\omega)$. Then by the local-global principle, we have
\[\left(\sqrt{-1}\right)^{\sigma_{1+2\omega_2}-1}=\hilbert{1+2\omega_2}{-1}{K_2}{2}=\hilbert{1+2\omega_v}{-1}{K_v}{2}^{-1}=(-1)^{-(3-1)/2}\mod (1+2\omega)=-1,\]
where $\hilbert{\cdot}{\cdot}{K_w}{2}$ denotes the second Hilbert symbol over $K_w$. Similarly,
\[\left(\sqrt[3]{2}\right)^{\sigma_{\omega_2}-1}=\hilbert{\omega_2}{2}{K_2}{3}=\omega^{-(4-1)/3}\mod 2=\omega^2.\]
\end{proof}

By Proposition \ref{LCF} and \ref{thm:Galois}, $\phi(P_1),  \phi(P_2)\in E_3(H_{3p})$. Let 
\begin{equation}\label{Heegnerpoint}
z_1=\Tr_{H_{3p}/L_{(p)}}\phi(P_1),\  \mathrm{and}\  z_2=\Tr_{H_{3p}/L_{(p)}}\phi(P_2),
\end{equation}
then $z_1,z_2\in E_3(L_{(p)})$.
\begin{thm}\label{nontrivial}
Both $z_1$ and $z_2$ are nontorsion.
\end{thm}
\begin{proof}
By Theorem \ref{thm:Galois}, 
\begin{equation}\label{action}
z_i^{\sigma_{\omega_3}}=[\omega^i]z_i+\frac{p+1}{3}(0, 36\sqrt{-3})
\end{equation}
for $i=1,2$. By \cite[Proposition 2.3]{SSY}, the torsion points in $E_3(L_{(p)})$ are $O,(0,\pm 36\sqrt{-3})$ which can not satisfy (\ref{action}).
\end{proof}
Let $\phi_p:E_3\ra E_{3p}$ and $\phi_{p^2}:E_3\ra E_{3p^2}$ be the map given by $(x,y)\mapsto (\sqrt[3]{p}x,py)$ and $(x,y)\mapsto (\sqrt[3]{p^2}x,p^2y)$. Set 
\[y_i=[\sqrt{-3}]z_i\in E_3(L_{(p)}).\]
By (\ref{action}), we know that $(y_i)^{\sigma_{\omega_3}}=[\omega^i](y_i)$. 
\begin{proof}[Proof of Theorem \ref{main}]

 By Proposition \ref{LCF} and Theorem \ref{nontrivial}, if $p\equiv 2\mod 9$, then $\phi_p(y_2)$ is a nontorsion point in $E_{3p}(K)$ and $\phi_{p^2}(y_1)$ is a nontorsion point in $E_{3p^2}(K)$; if $p\equiv 5\mod 9$, then $\phi_p(y_1)$ is a nontorsion point in $E_{3p}(K)$ and $\phi_{p^2}(y_2)$ is a nontorsion point in $E_{3p^2}(K)$. Since $E_{3p}(\BQ)$ has the same rank with $E_{3p}(K)$ and $E_{3p^2}(\BQ)$ has the same rank with $E_{3p^2}(K)$, we finish the proof.
\end{proof}
Define
\[E_3(L_{(p)})^{\sigma_{\omega_3}=[\omega^i]}:=\{P\in E_3(L_{(p)}):   P^{\sigma_{\omega_3}}=[\omega^i]P\}.\]
\begin{thm}\label{non-divisibility1}
The point $y_i$ are not divisible by $\sqrt{-3}$ in $E_3(L_{(p)})^{\sigma_{\omega_3}=[\omega^i]}$.
\end{thm}
\begin{proof}
Assume the contrary that $y_i=\sqrt{-3}Q+T$ with 
\[Q\in E_3(L_{(p)})^{\sigma_{\omega_3}=[\omega^i]} \textrm{ and } T\in E_3(L_{(p)})^{\sigma_{\omega_3}=[\omega^i]}_\tor.\]
Then
\[\sqrt{-3}(z_i-Q)=T.\]
Since $E_3(L_{(p)})_\tor=E_3[\sqrt{-3}]$, we conclude that $z_i-Q\in E_3[\sqrt{-3}]$. Suppose $z_i=Q+R$ with $R\in E_3[\sqrt{-3}]$. Taking Galois action of $\sigma_{\omega_3}$, we obtain 
$$0=[1-\omega^i]R=(0,\pm 36\sqrt{-3}),$$
which is a contradiction.
\end{proof}
\begin{remark}
As is seen  if $p\equiv 2\mod 9$ resp. $p\equiv 5\mod 9$, $y_1$ may be identified with a point of infinite order in $E_{3p}(K)$ resp. $E_{3p^2}(K)$; and  if $p\equiv 2\mod 9$ resp. $p\equiv 5\mod 9,$ $y_2$ may be identified with a point of infinite order in $E_{3p^2}(K)$ resp. $E_{3p}(K)$.
\end{remark}

\section{The explicit Gross-Zagier formulae}\label{E-G-Z}
In the rest of the paper we clarify the embedding of $K$ into $\M_2(\BQ)$ as follows. As indicated in the proof of Theorem \ref{main}, in the case $p\equiv 2 \mod 9$ and $\chi=\chi_{3p^2}$ or the case $p\equiv 5\mod 9$ and $\chi=\chi_{3p}$, we use the Heegner point $ z_1$ to construct nontrivial points on elliptic curves, and hence we embed $K$ into $\M_2(\BQ)$ under $\rho_1$.  Otherwise, in the case $p\equiv 2 \mod 9$ and $\chi=\chi_{3p}$ or the case $p\equiv 5\mod 9$ and $\chi=\chi_{3p^2}$, we use the Heegner point $ z_2$, and we embed $K$ into $\M_2(\BQ)$ under $\rho_2$. 
\subsection{The explicit Gross-Zagier formulae}
Let $\pi$ be the automorphic representation of $\GL_2(\BA)$ corresponding to ${E_9}_{/\BQ}$. Then $\pi$ is only ramified at $3$ with conductor $3^5$. For $n\in \BQ^\times$, let $\chi_n: \Gal(K^{\ab}/K)\rightarrow\BC^\times$
be the cubic character given by $\chi_n(\sigma)=(\sqrt[3]{n})^{\sigma-1}$. Define
\[L(s,E_9,\chi_n):=L(s-1/2,\pi_K\otimes \chi_n),\quad \epsilon(E_9,\chi_n):=\epsilon(1/2,\pi_K\otimes \chi_n),\]
where $\pi_K$ is the base change of $\pi$ to $\GL_2(\BA_K)$.

Let $p\equiv 2,5\mod 9$ be an odd prime number, and put $\chi=\chi_{3p}$ resp. $\chi_{3p^2}$. From the Artin formalism, we have
$$L(s,E_9,\chi)=L(s,E_p)L(s,E_{3p^2}) \text{ resp. } L(s,E_{p^2})L(s,E_{3p}).$$
By \cite{Liverance}, we have  the epsilon factors $\epsilon(E_{3p^2})(\text{ resp. } \epsilon(E_{3p}))=-1$ and $\epsilon(E_p)(\text{ resp. } \epsilon(E_{p^2}))=+1$, and hence the epsilon factor 
\[\epsilon(E_9,\chi)=\epsilon(E_{p})\epsilon(E_{3p^2})\text{ resp. }\epsilon(E_{p^2})\epsilon(E_{3p})=-1.\] 
For a quaternion algebra $\BB_{\BA}$, we define its ramification index $\epsilon(\BB_v)=+1$ for any place $v$ of $\BQ$ if the local component $\BB_v$ is split and $\epsilon(\BB_v)=-1$ otherwise. The following proposition guarentees we are in the same setting as in \cite[Theorem 4.3]{HSY}. 
\begin{prop}\label{Tun-Saito}
The incoherent quaternion algebra $\BB$ over $\BA$, which satisfies
$$\epsilon(1/2,\pi_{K,v}\otimes\chi_v)=\chi_{v}(-1)\epsilon_v(\BB)$$
for all places $v$ of $\BQ$, is only ramified at the infinity place.
\end{prop}
\begin{proof} Since $\pi$ is unramified at finite places $v\nmid 3$, $\chi$ is unramified at finite places $v\nmid 3p$ and $p$ is inert in $K$, by \cite[Proposition 6.3]{Gross88} we get $\epsilon(1/2,\pi_{K,v}\otimes\chi_v)=+1$ for all  finite $v\neq 3$. Again by \cite[Proposition 6.5]{Gross88}, we also know that $\epsilon(1/2,\pi_{K,\infty}\otimes\chi_\infty)=-1$. Since $\epsilon(1/2,\pi_K\otimes\chi)=-1$, we see that $\epsilon(1/2,\pi_{K,3}\otimes\chi_3)=+1$. Since $\chi$ is a cubic character, $\chi_v(-1)=1$ for any $v$. Hence $\BB$ is only ramified at the infinity place.
\end{proof}

Recall we have defined the Heegner points $z_1,z_2$ in (\ref{Heegnerpoint}). We also define 
\begin{equation}\label{z3}z_3=\Tr_{H_{3p}/L_{(p)}}\phi(\Tr_{H_{36p}/H_{9p}}P_3).\end{equation}
Since we use the same elliptic curve $E_9$ as in \cite[Theorem 4.3]{HSY}, very little modification of the proof gives us the following  explicit Gross-Zagier formulae once we verity the explicit local computation of toric integrals in Corollary \ref{ration}.
\begin{thm} \label{thm:GZ}
One has the following explicit formulae of Heegner points
\[\frac{L(1,E_{p})L'(1,E_{3p^2})}{\Omega_{p}\Omega_{3p^2}}=\left\{\begin{aligned} 2^{-1}\cdot 9 \cdot \wh{h}_\BQ(z_2),&\text{ if }p\equiv 2\mod 9;\\9\cdot\wh{h}_\BQ(z_1),&\text{ if }p\equiv 5\mod 9.\end{aligned}\right.\]
And
\[\frac{L(1,E_{p^2})L'(1,E_{3p})}{\Omega_{p^2}\Omega_{3p}}=\left\{\begin{aligned} 9 \cdot \wh{h}_\BQ(z_1),&\text{ if }p\equiv 2\mod 9;\\2^{-1}\cdot 9\cdot\wh{h}_\BQ(z_2),&\text{ if }p\equiv 5\mod 9.\end{aligned}\right.\]
\end{thm}

\begin{thm} \label{thm:GZ2}
One also has the following explicit formulae of Heegner points
\[\frac{L(1,E_{p})L'(1,E_{3p^2})}{\Omega_{p}\Omega_{3p^2}}= 2^{-3} \cdot \wh{h}_\BQ(z_3),\text{ if }p\equiv 2\mod 9\]
And
\[\frac{L(1,E_{p^2})L'(1,E_{3p})}{\Omega_{p^2}\Omega_{3p}}=\\2^{-3}\cdot\wh{h}_\BQ(z_3),\text{ if }p\equiv 5\mod 9.\]
\end{thm}
\begin{remark}
The defference between the formulae of $z_2$ and $z_3$ is due to the fact that we take a trace from $H_{36p}$ to $H_{9p}$ for $z_3$ which is of degree $6$. More explictely, we use $\frac{\#\Pic(\CO_p)}{\#\Pic(\CO_{36p})}$ instead of $\frac{\#\Pic(\CO_p)}{\#\Pic(\CO_{9p})}$ for $z_3$ when we modify the proof of \cite[Theorem 4.3]{HSY}.
\end{remark}
\begin{coro}\label{maincoro}
$z_3\in E_3(L(p))$ is nontorsion and satisfies $z_3^{\sigma_{\omega_3}}=[\omega^2]z_3$. If $p\equiv 2\mod 9$, then $\phi_p(z_3)$ is a nontorsion point in $E_{3p}(\BQ)$. If $p\equiv 5\mod 9$, then $\phi_p^2(z_3)$ is a nontorsion point in $E_{3p^2}(\BQ)$. In both cases,  $\phi_p(z_3)$ and $\phi_{p^2}(z_2)$ are not divisible by $3$ over $\BQ$.
\end{coro}
\begin{proof}
Since $[H_{36p}:H_{9p}]=6$, by Theorem \ref{thm:Galois} and (\ref{z3}), we know that $z_3^{\sigma_{\omega_3}}=[\omega^2]z_3$. Since $z_3$ is a real point by Remark \ref{remark1}, by Proposition \ref{LCF}, if $p\equiv 2\mod 9$, $\phi_p(z_3)\in E_{3p}(\BQ)$ and if $p\equiv 5\mod 9$, $\phi_{p^2}(z_3)\in E_{3p^2}(\BQ)$.
By Theorem \ref{thm:GZ} and Theorem \ref{thm:GZ2}, $\wh{h}_\BQ(z_3)=36\wh{h}_\BQ(z_2)=12\wh{h}_\BQ(y_2)$. This implies $z_3$ is nontorsion and $\phi_p(z_3)$, $\phi_{p^2}(z_3)$ can not be divisible by $3$ over $\BQ$. Otherwise there exists point $z$ in $E_3(L(p))^{\sigma_{\omega_3}=[\omega^2]}$ such that $9\wh{h}_\BQ(z)=\wh{h}_\BQ(z_3)=12\wh{h}_\BQ(y_2)$. But this is impossible since $y_2$ is not divisible by $\sqrt{-3}$ in $E_3(L_{(p)})^{\sigma_{\omega_3}=[\omega^2]}$ and $E_3(L(p))^{\sigma_{\omega_3}=[\omega^2]}$ is of rank $1$ over $K$ by Kolyvagin.
\end{proof}

\subsection{Waldspurger's local period integral}
This subsection is purely local and we shall compute the $3$-adic period integral  for the $3$-adic local newform following \cite{HSY2}. Recall $\pi$ is the automorphic representation of $\GL_2(\BQ)$ corresponding to $E_9$ and $\pi_3$ the $3$-adic part of $\pi$. Then the conductor $c(\pi_3)=3^5$. Let $p\equiv 2,5\mod 9$ be an odd prime and let  $\chi:\Gal(\bar{K}/K)\ra \CO_K^\times$ be the character given by $\chi(\sigma)=\chi_{3p}(\sigma)=(\sqrt[3]{3p})^{\sigma-1}$  resp. $\chi(\sigma)=\chi_{3p^2}(\sigma)=(\sqrt[3]{3p^2})^{\sigma-1}$. We also view $\chi$ as a Hecke character on $\BA_K^\times$ by the Artin map and the conductor the 3 part is $c(\chi_3)=(\sqrt{-3})^4$. Assume that $f_3$ is the standard newform of $\pi_3$.  
We shall compute the following normalized period integral
\begin{equation}\label{wpi}
\beta^0_3(f_3, f_3)
=\int\limits_{t\in \BQ_3^\times\backslash K_3^\times}\frac{ (\pi(t)f_3,f_3)}{(f_3,f_3)}\chi_3(t)dt
\end{equation}
which appears in the proof of the explicit Gross-Zagier formulae.
Let $\Theta:K^\times\backslash \BA_K^\times\ra \BC^\times$ be the unitary Hecke character associated to the base-changed CM elliptic curve ${E_9}_{/K}$. 
Then $\Theta$ has conductor $9\CO_K$.  We denote $\Theta_3$ the 3-adic component of $\Theta$. Then $\pi_3$ is the local representation of $\GL_2(\BQ_3)$ corresponding to $\Theta_3$. Note
\[\CO_{K,3}^\times/(1+9\CO_{K,3})\simeq \langle \pm 1\rangle ^{\BZ/2\BZ} \times\langle 1+\sqrt{-3}\rangle^{\BZ/3\BZ}\times\langle 1-\sqrt{-3}\rangle^{\BZ/3\BZ}\times \langle 1+3\sqrt{-3}\rangle^{\BZ/3\BZ}.\]
\begin{lem}\label{thetavalue}
We have $c(\Theta_3)=4$, and $\Theta_3$ is
given explicitly by
\[\Theta_3(-1)=-1,\quad \Theta_3(1+\sqrt{-3})=\frac{-1-\sqrt{-3}}{2},\ \ \Theta_3(\sqrt{-3})=i,\]
\[\Theta_3(1-\sqrt{-3})=\frac{-1+\sqrt{-3}}{2},\quad \Theta_3(1+3\sqrt{-3})=\frac{-1+\sqrt{-3}}{2}.\]
\[\]
\end{lem}
\begin{proof}
This is \cite[Lemma 4.1]{HSY2}.
\end{proof}
The local  character  $\chi_3$ has conductor $\BZ_3^\times(1+9\CO_{K,3})$, and hence it is in fact a character of the quotient group $K_3^\times/\BQ_3^\times(1+9\CO_{K,3})$. Note that
\[K_3^\times/\BQ_3^\times(1+9\CO_{K,3})\simeq \langle\sqrt{-3}\rangle^{\BZ}\times\langle 1+\sqrt{-3}\rangle^{\BZ/3\BZ}\times\langle 1+3\sqrt{-3}\rangle^{\BZ/3\BZ}.\]
\begin{lem} \label{chi}
We have $c(\chi_3)=4$ and $\chi_3$ is given explicitly by the following tables:
\begin{itemize}
\item[1.] if $\chi=\chi_{3p}$, then
\begin{center}
\begin{tabular}{|c|c|c|c|c|c|}
\hline
$p\mod 9$&$\chi_3(1+\sqrt{-3})$&$\chi_3(1+3\sqrt{-3})$&$\chi_3(\sqrt{-3})$\\
\hline
$2$&$\omega^2$&$\omega$&$1$\\
\hline
$5$&$\omega$&$\omega$&$1$\\
\hline 
\end{tabular}.
\end{center}
\item[2.] if $\chi=\chi_{3p^2}$, then
\begin{center}
\begin{tabular}{|c|c|c|c|c|c|}
\hline
$p\mod 9$&$\chi_3(1+\sqrt{-3})$&$\chi_3(1+3\sqrt{-3})$&$\chi_3(\sqrt{-3})$\\
\hline
$2$&$\omega$&$\omega$&$1$\\
\hline
$5$&$\omega^2$&$\omega$&$1$\\
\hline 
\end{tabular}.
\end{center}
\end{itemize}
\end{lem}

\begin{proof}
The proof is routine in class-field theory. See \cite[Lemma 4.2]{HSY2} for more details.
\end{proof}

\begin{coro}
\label{thetachivalue}
If $p\equiv 2 \textrm{ resp. } 5 \mod 9$, and $\chi=\chi_{3p}\textrm{ resp. }\chi_{3p^2}$,
then the local character $\Theta_3\ov\chi_3$ is given explicitly by
\[\Theta_3\ov\chi_3(-1)=-1,\quad \Theta_3\ov\chi_3(1+\sqrt{-3})=1,\]
\[\Theta_3\ov\chi_3(1-\sqrt{-3})=1,\quad \Theta_3\ov\chi_3(1+3\sqrt{-3})=1,\quad \Theta_3\ov\chi_3(\sqrt{-3})=i.\]
If $p\equiv 2 \textrm{ resp. } 5 \mod 9$, and $\chi=\chi_{3p^2}\textrm{ resp. }\chi_{3p}$,
the local character $\Theta_3\ov\chi_3$ is given explicitly by
\[\Theta_3\ov\chi_3(-1)=-1,\quad \Theta_3\ov\chi_3(1+\sqrt{-3})=\omega,\]
\[\Theta_3\ov\chi_3(1-\sqrt{-3})=\omega^2,\quad \Theta_3\ov\chi_3(1+3\sqrt{-3})=1,\quad \Theta_3\ov\chi_3(\sqrt{-3})=i.\]
\end{coro}

Let $\theta_3$ be the $3$-adic character which parametrizes the supercuspidal representation $\pi_3$ via compact-induction construction as in \cite[Section 2.2]{HSY2}. The test vector issue for Waldspurger's period integral is closely related to $c(\theta_3\ov\chi_3)$ or $c(\theta_3\chi_3)$. We can work out these  by using Lemma \ref{thetavalue}, \ref{chi} and Corollary \ref{thetachivalue}, and the relation between $\theta_3$ and $\Theta_3$ in \cite[Theorem 2.8]{HSY2}.
Now we can prove the following key Lemma.
\begin{lem}\label{Cor:AllnecessaryFormulationForThetaChi}
 If $p\equiv 2 \textrm{ resp. } 5 \mod 9$, and $\chi=\chi_{3p}\textrm{ resp. }\chi_{3p^2}$, we have $\theta_3\ov\chi_3=1$. If $p\equiv 2 \textrm{ resp. } 5 \mod 9$, and $\chi=\chi_{3p^2}\textrm{ resp. }\chi_{3p}$, we have $c(\theta_3\ov\chi_3)=2$ and $\alpha_{\theta_3\ov\chi_3}=\frac{1}{3\sqrt{-3}}$. Moreover, in any cases, $c(\theta_3\ov{\chi}_3)\leq c(\theta_3\chi_3)$.
\end{lem}
\begin{proof}
Let $\psi_3$ be the additive character such that $\psi_3(x)=e^{2\pi i \iota(x)}$ where $\iota:\BQ_3\rightarrow \BQ_3/\BZ_3 \subset \BQ/\BZ$ is the map given by $x\mapsto -x\mod \BZ_3$ which is compatible with the choice in \cite{CST14}. Let $\psi_{K_3}(x)=\psi_3\circ \Tr_{K_3/\BQ_3}(x)$, be the additive character of $K_3$. 

Recall that $\alpha_{\Theta_3}$ is the number associated to $\Theta_3$ as in \cite[Lemma 2.1]{HSY2} so that
\[\Theta_3(1+x)=\psi_{K_3}(\alpha_{\Theta_3} x), \]
for any $x$ satisfying $v_{K_3}(x)\geq c(\Theta_3)/2=2$. By the definition of $\psi_{K_3}$ and Lemma \ref{thetavalue}, we know that $\alpha_{\Theta_3}=\frac{1}{9\sqrt{-3}}$.  Now let $\eta_3$ be the quadratic character associated to the quadratic field extension $K_3/\BQ_3$.
Then by \cite[Proposition 34.3]{BushnellHenniart:06a}, $\lambda_{K_3/\BQ_3}(\psi')=\tau(\eta_3,\psi_3')/\sqrt{3}=-i$, here $\tau(\eta_3,\psi_3')$ is the Gauss sum and $\psi_3'(x)=\psi_3(\frac{x}{3})$ is the additive character of level one. By \cite[Lemma 5.1]{Langlands}, $\lambda_{K_3/\BQ_3}(\psi_3)=\eta_3(3)\lambda_{K_3/\BQ_3}(\psi_3')=-i$. Then $\Delta_{\theta_3}$ is the unique level one character of $K_3$ such that $\Delta_{\theta_3}|_{\BZ_3^\times}=\eta_3$ and
\[\Delta_{\theta_3}(\sqrt{-3})=\eta((\sqrt{-3})^3\alpha_{\Theta_3})\lambda_{K_3/\BQ_3}(\psi_3)^{3}=-i.\]
Recall that $\theta_3=\Theta_3\Delta_{\theta_3}$. Then by Corollary \ref{thetachivalue} we can easily check that: 

\begin{enumerate}
\item  If $p\equiv 2 \textrm{ resp. } 5 \mod 9$, and $\chi=\chi_{3p}\textrm{ resp. }\chi_{3p^2}$, $\theta_3\ov\chi_3$ is the trivial character. 
\item If $p\equiv 2 \textrm{ resp. } 5 \mod 9$, and $\chi=\chi_{3p^2}\textrm{ resp. }\chi_{3p}$, $\theta_3\ov\chi_3$ is of level 2 and by definition we can choose $\alpha_{\theta_3\ov\chi_3}=\frac{1}{3\sqrt{-3}}$.
\end{enumerate}
Since $\chi_3$ is a cubic character, $\theta_3\chi_3=\theta_3\ov{\chi}_3^2$. Since $c(\chi_3)=c(\ov{\chi}_3)=4$, $c(\theta_3\chi_3)=4$ and the last assertion follows.
\end{proof}

To apply the results in \cite{HSY2} to calculate the local period integral, we take $\BF=\BQ_3$,  $\varpi=3=q$, $D=-3$, $K_3\simeq\BE\simeq\BL\simeq\BQ(\sqrt{-3})_3$, $c(\theta_3)=c(\chi_3)=4$, $n=2$. By \cite[Lemma 2.9]{HSY2}, we have the minimal vector $\varphi_0=\Char(\varpi^{-2}U_\BF(1))$ in the Kirillov model.  Recall $K$ is embedded into $\rm{M}_2(\BQ)$ as in Section 2.2 which linearly extends the following map:
\begin{equation}\label{emb}
\sqrt{-3}\mapsto \matrixx{a}{3^{-2}b}{3^3c}{-a}:=\begin{cases}
 \matrixx{3}{-2p/9}{54/p}{-3}, &\text{\ if $K$ is embedded under $\rho_1$;}\\
  \matrixx{9}{-2p/9}{374/p}{-9}, &\text{\ if $K$ is embedded under $\rho_2$;}\\
 \matrixx{0}{-p/18}{54/p}{0}, &\text{\ if $K$ is embedded under $\rho_3$.}
 \end{cases}
\end{equation}

\begin{prop}\label{Prop:TestingForNew}
Suppose $\Vol(\BZ_3^\times\backslash\CO_{K,3}^\times)=1$ so that $\Vol(\BQ_3^\times\backslash K_3^\times)=2$.
For $f_3$ being the newform, $K$ being embedded in $\M_2(\BQ)$ as in $(\ref{emb})$, we have
\begin{equation}
 \beta^0_3(f_3,f_3)=\begin{cases}
 1, &\text{\ if $p\equiv 2 \textrm{ resp. } 5 \mod 9$, $\chi=\chi_{3p}\textrm{ resp. }\chi_{3p^2}$ and $K$ is embedded under $\rho_2$ or $\rho_3$;}\\
 1/2, &\text{\ if $p\equiv 2 \textrm{ resp. } 5 \mod 9$, $\chi=\chi_{3p^2}\textrm{ resp. }\chi_{3p}$ and $K$ is embedded under $\rho_1$. } \end{cases}
\end{equation}
\end{prop}

\begin{proof}
We may assume $f_3$ to be $L^2$-normalized. To evaluate $f_3$ for the embedding in \eqref{emb} is equivalent to use the standard embedding \cite[(2.13)]{HSY2} of $\BE$ and use the corresponding translate of the newform.
In particular the embedding in \eqref{emb} is conjugate to the standard embedding by the following 
\begin{equation}
    \matrixx{a}{3^{-2}b}{3^3c}{-a}=\matrixx{-9c}{a/3}{0}{1}^{-1}\matrixx{0}{1}{D}{0}\matrixx{-9c}{a/3}{0}{1}.
\end{equation}
Thus  we have
\begin{align}
\beta^0_3(f_3,f_3)&=\int\limits_{\BF^\times\backslash \BE^\times}\left(\pi_3\lrb{\matrixx{-9c}{a/3}{0}{1}^{-1}t\matrixx{-9c}{a/3}{0}{1}}f_3,f_3\right)\chi(t)dt\\
&=\int\limits_{\BF^\times\backslash \BE^\times}\left(\pi_3\lrb{t\matrixx{-9c}{a/3}{0}{1}}f_3,\pi_3\lrb{\matrixx{-9c}{a/3}{0}{1}f_3}\right)\chi(t)dt\notag,
\end{align}
which is by definition $\BetaI{\pi_3\lrb{\matrixx{-9c}{a/3}{0}{1}}f_3,\pi_3\lrb{\matrixx{-9c}{a/3}{0}{1}f_3}}$ for the bilinear pairing as in \cite[(3.1)]{HSY2} and the standard embedding as in \cite[(2.13)]{HSY2}.
Note that by \cite[Corollary 2.10]{HSY2},
$$\pi_3\lrb{\matrixx{-9c}{a/3}{0}{1}}f_3
=\frac{1}{\sqrt{2}}\sum\limits_{x\in (\BZ_3/3 \BZ_3)^\times}\pi_3\lrb{\matrixx{1}{a/3}{0}{1}\matrixx{x}{0}{0}{1}}\varphi_0$$
where $\varphi_0$ is the minimal test vector.

If $p\equiv 2$ resp.  $5 \mod 9$ and $\chi=\chi_{3p}$ resp. $\chi_{3p^2}$, we embed $K$ into $\M_2(\BQ)$ under $\rho_2$ or $\rho_3$. In this case, $c(\theta_3\ov{\chi_3})=0$ and $a=0$ or $9$. By the $l=0$ case in \cite[Section 2.4]{HSY2}, we have a unique $x\mod\varpi$ for which $\BetaI{\varphi_x,\varphi_x}$ is nonvanishing (In fact, we must have $x\equiv 1\mod 3$).  According to \cite[Proposition 3.3]{HSY2},
 there are no off-diagonal terms, and we have
\begin{equation}
\beta^0_3\lrb{f_3,f_3 }=\frac{1}{(q-1)q^{\lceil \frac{c(\theta_3)}{2 e_\BL}\rceil-1}} \BetaI{\varphi_x,\varphi_x}=\frac{1}{2}\cdot 2=1.
\end{equation}

If $p\equiv 2 \textrm{ resp. } 5 \mod 9$, and $\chi=\chi_{3p^2}\textrm{ resp. }\chi_{3p}$, we embed $K$ into $\M_2(\BQ)$ under $\rho_1$. In this case, we have $c(\theta_3\ov{\chi}_3)=2l=2$ and $u=a/3=1$. The action of $\matrixx{1}{1}{0}{1}$ on $\varphi_x=\pi_3\lrb{\matrixx{x}{0}{0}{1}}\varphi_0$ is by a simple character. This is the case $l=1$ and $n-l=1$ is odd. By the choice in \cite[Section 2.4]{HSY2}, 
$$D'=\frac{1}{\alpha_{\theta_3}^2\varpi_\BL^{2c(\theta_3)}}=-3,$$
noting $\alpha_{\theta_3}=\alpha_{\Theta_3}=\frac{1}{9\sqrt{-3}}$.

By Lemma \ref{Cor:AllnecessaryFormulationForThetaChi}, $\alpha_{\theta_3\chi_3^{-1}}=\frac{1}{3\sqrt{-3}}$ in this case, and we have 
\begin{align}
\Delta(1)&= 4\varpi^{n}\alpha_{\theta_3\ov\chi_3}\sqrt{D}\lrb{\varpi^{n}\alpha_{\theta_3\ov\chi_3}\sqrt{D}-2\sqrt{\frac{D}{D'}}}+4\frac{D}{D'}D \\
&\equiv 4\cdot 9\cdot \frac{1}{3\sqrt{-3}} \cdot \sqrt{-3}\cdot (-2)+4 \cdot (-3)\mod{\varpi^2}\notag\\
&\equiv  -8\cdot 3-4\cdot 3 \mod{\varpi^2}\notag\\
&\equiv 0 \mod{\varpi^2}\notag.
\end{align}
$\Delta(1)$ is indeed  congruent to a square. Then we can get a unique solution of $x\mod \varpi$ from \cite[(2.17)]{HSY2}, and again by \cite[Proposition 3.3]{HSY2},
\begin{equation}
 \beta^0_3\lrb{f_3,f_3}=\frac{1}{(q-1)q^{\lceil \frac{c(\theta_3)}{2 e_\BL}\rceil-1}}\frac{1}{q^{\lfloor l/2\rfloor}}=\frac{1}{2}.
\end{equation}
\end{proof}

Let $f'$ be the admissible test vector of $(\pi, \chi)$ which is as defined in \cite[Definition 1.4]{CST14}. By definition, the $3$-adic part $f_3'$ is $\chi_3^{-1}$-eigen under the action of $K_3^\times$. The following corollary is used in the proof of the explicit Gross-Zagier formulae.
\begin{coro}\label{ration}
For the admissible test vector $f'_3$ and the newform $f_3$ we have 
\[\frac{\beta_3^0(f'_3,f'_3)}{\beta_3^0(f_3,f_3)}=\begin{cases}
 2, &\text{\ if $p\equiv 2 \textrm{ resp. } 5 \mod 9$, $\chi=\chi_{3p}\textrm{ resp. }\chi_{3p^2}$ and $K$ is embedded under $\rho_2$ and $\rho_3$,}\\
 4, &\text{\ if $p\equiv 2 \textrm{ resp. } 5 \mod 9$, $\chi=\chi_{3p^2}\textrm{ resp. }\chi_{3p}$ and $K$ is embedded under $\rho_1$}. 
 \end{cases}\]
\end{coro}
\begin{proof}
Keep the normalization of the volumes in Proposition \ref{Prop:TestingForNew}. By definition of $f'$, we have $\beta_3^0(f'_3,f'_3)=\Vol(\BQ_3^\times\backslash K_3^\times)=2$. Then the corollary follows from Proposition \ref{Prop:TestingForNew}.
\end{proof}

\section{The $3$-part of the Birch and Swinnerton-Dyer conjectures}

Let $n$ be a positive cube-free integer and $E'_n$ be the elliptic curve given by Weierstrass equation $y^2=x^3+(4n)^2$. Then there is an unique isogeny
$\phi_n: E_n\rightarrow E'_n$ of degree $3$ up to $[\pm 1]$ and denote $\phi_n'$ its dual isogeny.

\begin{prop}\label{sha}
Let $p\equiv 2,5\mod 9$ be an odd prime. Then
$$\dim_{\BF_3}\Sel_3(E_{3p^2}(\BQ))\leq 1,\quad \dim_{\BF_3}\Sel_3(E_{p}(\BQ))=0.$$
$$\dim_{\BF_3}\Sel_3(E_{3p}(\BQ))\leq 1,\quad \dim_{\BF_3}\Sel_3(E_{p^2}(\BQ))=0.$$
\end{prop}
\begin{proof}
By \cite[Theorem 2.9]{Satge}, we know that
$$\Sel_{\phi_{3p^2}}(E_{3p^2}(\BQ))=\Sel_{\phi_{3p^2}'}(E'_{3p^2}(\BQ))=\BZ/3\BZ,$$
and
$$\Sel_{\phi_{p}}(E_{3p^2}(\BQ))=\BZ/3\BZ,\quad \Sel_{\phi_{p}'}(E'_{3p^2}(\BQ))=0.$$
Note that $E_p[3](\BQ)$ and $E_{3p^2}[3](\BQ)$ are trivial and $|E'_p[\phi'_p](\BQ)|=|E'_{3p^2}[\phi'_{3p^2}](\BQ)|=3$. By \cite[Lemma 5.1]{HSY}, we have $$\dim_{\BF_3}\Sel_3(E_{3p^2}(\BQ))\leq 1,\quad \dim_{\BF_3}\Sel_3(E_{p}(\BQ))=0.$$ Similarly we have $$\dim_{\BF_3}\Sel_3(E_{3p}(\BQ))\leq 1,\quad \dim_{\BF_3}\Sel_3(E_{p^2}(\BQ))=0.$$
\end{proof}

Now we are ready to give the proof of Theorem \ref{main2}.
\begin{proof}[Proof of Theorem \ref{main2}]
We will give the proof of (\ref{bsd1}) when $p\equiv 2\mod9$.  One can verify (\ref{bsd2}) in case $p\equiv 5\mod9$ similarly.  Now assume $p\equiv 2\mod9$.
By \cite[Table 1]{ZK}, we know that $c_3(E_p)=2$ and $c_\ell(E_p)=1$ for any prime $\ell\neq 3$, while $c_\ell(E_{3p^2})=1$ for all primes $\ell$.  

Let $P$ be the generator of the free part of $E_{3p^2}(\BQ)$. Then the BSD conjecture predicts that
\begin{equation*}
\frac{L'(1,E_{3p^2})}{\Omega_{3p^2}}=|\Sha(E_{3p^2})|\cdot\widehat{h}_\BQ(P)\ \ \ \mathrm{and}\ \ \ 
\frac{L(1,E_{p})}{\Omega_{p}}=2\cdot\left|\Sha(E_{p})\right|.
\end{equation*}
Combining these two, we get
\begin{equation}\label{bsd}
\frac{L(1,E_p)}{\Omega_p}\cdot\frac{L'(1,E_{3p^2})}{\Omega_{3p^2}}=2\cdot|\Sha(E_p)|\cdot|\Sha(E_{3p^2})|\cdot\widehat{h}_\BQ(P).
\end{equation}
By Theorem \ref{thm:GZ2}, we expect
\begin{equation}\label{5.2}
|\Sha(E_p)|\cdot|\Sha(E_{3p^2})|=2^{-4}\cdot\frac{\widehat{h}_\BQ(z_3)}{\widehat{h}_\BQ(P)}.
\end{equation}
Note the RHS of (\ref{5.2}) is a nonzero rational number.

By Proposition \ref{sha},  $E_{3p^2}(\BQ)$ has rank 1, and form the exact sequence
\[0\ra E(\mathbb{Q})/3E(\mathbb{Q})\ra Sel_3(E(\mathbb{Q}))\ra \Sha(E)[3]\ra 0,\]
we know directly that 
$$|\Sha(E_p)[3^\infty]|=|\Sha(E_{3p^2})[3^\infty]|=1.$$
In order to prove the $3$-part of (\ref{5.2}), it suffices to prove
$$\widehat{h}_\BQ(P)=u\cdot\widehat{h}_\BQ(z_3)$$
for some $u\in \BZ_3^\times \cap \BQ$. This is clear by Corollay \ref{maincoro}.

\end{proof}

\subsection*{ Acknowledgement}
Part of this paper is finished during the author Hongbo Yin's one year stay (2019-2020) in Max-Planck Institute for Mathematics, Bonn. He is grateful to its hospitality and financial support.

\bibliographystyle{alpha}
\bibliography{reference}
\end{document}